\documentclass[12pt,leqno]{amsart}
\usepackage{amsmath,amsfonts,amsthm,latexsym,graphicx,amssymb,url, color}
\usepackage{hyperref}
\usepackage{txfonts} 
\newcommand{\average}{{\mathchoice {\kern1ex\vcenter{\hrule
height.4pt width 6pt depth0pt} \kern-9.7pt}
{\kern1ex\vcenter{\hrule height.4pt width 4.3pt depth0pt}
\kern-7pt} {} {} }}

\newcommand{\dist}{\text{dist}} 
\newcommand{\spt}{\text{spt}}

\newcommand{\R}{{\mathbb R}} 

\newcommand{\N}{{\mathbb N}}

\newcommand{\Be}{B_{\text{ext}}}
\newcommand{\psie}{\psi_{\text{ext}}}

{\begin{list}{}%
         {\setlength{\leftmargin}{#1}}%
         \item[]%
}
{\end{list}}
\theoremstyle{plain}
\newtheorem{theorem}{Theorem}[section]
\newtheorem{corollary}[theorem]{Corollary}
\newtheorem{lemma}[theorem]{Lemma}

\newtheorem{remark}[theorem]{Remark}

\begin{document}

\title[Global magnetic confinement for 1.5D RVM systems]
{Global magnetic confinement for the 1.5D  Vlasov-Maxwell system}

\thanks{Research of T.T. Nguyen and W.A. Strauss was supported in part by the NSF under grants DMS-1405728 and DMS-1007960, respectively. Research of T.V. Nguyen was partially supported by the Simons Foundation under grant \# 318995.
}

\author{Toan T. Nguyen}
\address{Department of Mathematics, Pennsylvania State University, State College, PA 16802, USA}
\email{nguyen@math.psu.edu}
\author{Truyen V. Nguyen}
\address{Department of Mathematics, The University of Akron, Akron, OH 44325, USA}
\email{tnguyen@uakron.edu}
\author{Walter A. Strauss}
\address{Department of Mathematics and Lefschetz Center for Dynamical Systems,
Brown University, Providence, RI 02912, USA.}
\email{wstrauss@math.brown.edu.}


\begin{abstract} We establish the global-in-time existence and uniqueness  of classical solutions 
to the ``one and one-half'' dimensional 
relativistic Vlasov--Maxwell  systems in a bounded interval, subject to an external magnetic field which is 
infinitely large at the spatial boundary. We prove that 
the large external magnetic field confines the particles to a compact set away from the boundary. 
This excludes the known singularities that typically occur due to particles that repeatedly bounce off the boundary. 
In addition to the confinement, we follow the techniques introduced by Glassey and Schaeffer,  
who studied the Cauchy problem without boundaries. 
\end{abstract}

\maketitle

\setcounter{equation}{0}

\section{Introduction}
 
 Using external magnetic fields to confine plasmas has been one of major goals of fusion energy research. 
 It is one of the most promising mechanisms for producing safe new sources of fusion energy. 
 Scientists are particularly interested in designing stable devices to induce confinement (e.g., \cite{Ga,Wh}). 
  In this paper we establish {\em global-in-time magnetic confinement of a collisionless plasma},  
 albeit under an assumption of low dimension.  
 
Specifically, we consider the relativistic 
 Vlasov-Maxwell (RVM) system, subject to an external magnetic field $\Be$ in a bounded interval $\Omega=(0,1)$.  We assume a single species of particles with a nonnegative distribution function $f(t, x, v)$,
 where $t\geq 0$, $x\in \Omega$ and $v\in \R^2$. 
 In this $1\frac12$ dimensional model the Vlasov equation is  
  \begin{equation}\label{VE}
 \partial_t f + \hat v_1 \partial_x f + (E_1 + \hat v_2 \check B)\partial_{v_1}f + (E_2-\hat v_1 \check B)\partial_{v_2} f =0,
 \end{equation}
 where $\check B = B(t,x) +\Be(x)$ with $\Be(x)$ is a stationary external magnetic field that becomes infinitely large on the boundary.  The internal electric and magnetic field with components $E_1(t,x),$ $E_2(t,x), B(t,x)$ satisfies 
 the $1\frac12$D Maxwell equations 
 \begin{equation}\label{ME}
\left\{\begin{array}{rl}
\partial_t E_1 &= -  j_1; \quad \partial_x E_1 = \rho; \\
\partial_t E_2 &= -\partial_x B - j_2,\\
\partial_t B &= - \partial_x E_2. 
\end{array}\right.
\end{equation}
For mathematical simplicity all the physical constants have been normalized.  
 In this relativistic case, the velocity is $\hat v= (\hat v_1, \hat v_2) = {v}/{\sqrt{1+ |v|^2}}$.  
 The charge density $\rho$ and the current density $j=(j_1, j_2)$ are  
  \begin{align*}
  \rho(t,x) := \int_{\R^2}{f(t,x,v)\, dv}\quad \mbox{and}\quad
   j(t,x) := \int_{\R^2}{\hat v f(t,x,v)\, dv}.
  \end{align*}
We impose the standard  initial  conditions for the distribution function and the field, namely, 
\begin{equation}\label{initial-condition}
 f(0,x,v) = f^0(x,v)\geq 0, \quad E_2(0,x) = E_2^0(x), \quad B(0,x) = B^0(x),
\end{equation}
while the initial value for $E_1$ is already determined by means of the identity $\partial_x E_1 = \rho$ 
and the specification
\begin{equation}\label{initial-E_1}
 E_1(0,0) = \lambda
\end{equation}
for a given constant $\lambda\in \R$.

The novelty of this paper lies in the boundary conditions.  We assume 
\begin{equation}\label{boundary-condition}
E_2(t,x)_{\vert_{\partial \Omega}} = E_2^b(t,x) , \qquad B(t,x)_{\vert_{\partial \Omega}}  = B^b(t,x), 
\end{equation}
where $E_2^b(t,\cdot)$, $B^b(t,\cdot)$ are given functions  defined on the boundary.  
In the sequel we will show that no particle trajectory  can reach the boundary $\partial \Omega$ if it begins away from it. Because the particle density $f(t,x,v)$ is constant along each particle trajectory, {\it no boundary condition 
is needed for} $f(t,x,v)$, assuming that its initial support does not meet the boundary. 

Throughout the paper we take $\Be = \partial_x \psie(x)$, in which the potential function $\psie(x)$ is assumed to satisfy:   
\begin{equation}\label{blow-up-cond}
 \psie\in C^2(\Omega)\quad  \mbox{and}\quad  | \psie(x)|\geq  
\frac{c_0}{\dist(x,\partial \Omega)^{\gamma}}-\frac{1}{c_0}\quad \quad \forall ~x\in \Omega
\end{equation}
for some constants $\gamma>0$ and $c_0>0$. In particular, $\psie(x) = \infty$ on the boundary!

We are interested in the well-posedness of  the initial-boundary value problem \eqref{VE}--\eqref{boundary-condition}. In what follows, $ C^1(U)$ denotes the standard $C^1$ function space, and $C^1_0(U)$ consists of functions in $C^1(U)$ that have compact support in $U$.   
In particular,  $f\in C^1_0([0,T]\times \Omega\times \R^2 )$  means that 
 $f$ has compact support in the $(x,v)$-variable, but has no restriction in the $t$-variable.
 We now state our main result.  
\begin{theorem}     [Global well-posedness]   \label{thm:main}  
Assume that $f^0\in C^1_0(\Omega\times \R^2)$ is nonnegative, $\lambda\in\R$ and 
$E_2^0,\ B^0,\ E_2^b,\ B^b \in C^1$.   Assume also that  the external magnetic field $\Be  = \partial_x \psie$ satisfies \eqref{blow-up-cond}. Then the problem \eqref{VE}--\eqref{boundary-condition} has a unique global-in-time $C^1$ solution $(f, E_1, E_2, B)$. Moreover,  $f$ is nonnegative and  $f\in C^1_0([0,T]\times \Omega\times \R^2 )\,$ for any $T>0$. 
\end{theorem}

Let us mention a few previous results on the global Cauchy problem for the Vlasov--Maxwell system. 
It is well known that 
global {\it weak} solutions exist in the whole three-dimensional space (\cite{DiL}), 
even in the presence of boundaries (\cite{Gu1,M}).      
However,  it is a famous open problem as to whether such solutions are unique or regular. 
Concerning classical (smooth) solutions, the authors in \cite{GStr} established the global 
theory for RVM systems in the whole three-dimensional space under an assumption on the 
momentum support of the density.  Alternative proofs have been given in  \cite{BGP,KS}.  
Subsequently, there was a series of papers \cite{GSc,GSc2,GSc2.5} where the (unconditional) well-posedness and regularity of solutions were established for the $1\frac 12$, $2$, and $2\frac 12$ dimensional RVM system. 
The present paper is motivated by \cite{GSc}, our novelty being the presence of a boundary.  


There have been just a few mathematical studies of the magnetic confinement problem (e.g., \cite{HK,CCM1,CCM2}).  
All these papers are concerned with a plasma with no internal magnetic field 
but confined by an external magnetic field.  
In \cite{HK} further assumptions are introduced that reduce the problem to a system for the macroscopic density 
and electric field.  
In \cite{CCM1,CCM2}  Vlasov-Poisson systems with bounded  and unbounded charges are considered and an existence-uniqueness theorem is proved.  

When confining a plasma modeled by RVM  to a spatial domain, singularities are typically created   
at the boundary and they propagate inside the domain.  
This is true even for Vlasov-Poisson (VP) systems (i.e., without magnetic fields); see, e.g.,  \cite{Gu2}.   
Furthermore, some particles repeatedly bounce off the boundary, making it extremely difficult to analyze their trajectories. To the best of our knowledge, there is no global theory of classical $C^1$ solutions to the RVM systems in domains with boundaries, even for the simplest RVM model,  the $1\frac 12$ dimensional system \eqref{VE}--\eqref{ME} without
an external magnetic field.   

However, in our problem with a very intense external magnetic field at the boundary, singularities can be avoided 
because the particles that come near the boundary are drifted back into the plasma domain. 
Rigorous details of the confinement are provided in Section~\ref{sec:confinement}. The proof of our 
main theorem then follows along the lines of \cite{GSc}.


\section{Bounds on the field}\label{sec:bound-fields}
The proof of Theorem \ref{thm:main} relies on uniform {\it a priori} estimates. Let us consider a  $C^1$  solution $(f, E_1, E_2, B)$ of the RVM equations \eqref{VE}--\eqref{boundary-condition} on a finite time interval $[0,T]$ so that $f(t,x,v) =0$ at the boundary $x = 0,1$. We shall derive $L^\infty$ estimates  for the fields of such  solution. For convenience, we rewrite \eqref{VE} as
 \begin{equation}\label{VE'}
 \partial_t f + \hat v_1 \partial_x f + K\cdot \nabla_v f =0
 \end{equation} 
 with $K := 
 E + (\hat v_2, -\hat v_1)\check B$. Hereafter  we use $E$  to denote  the vector $(E_1, E_2)$.

\subsection{Estimate of $E_1$}
Integrating the Vlasov equation \eqref{VE'} in $v$ and using  $K\cdot \nabla_v f= \nabla_v \cdot (K f) $, 
we obtain
\begin{equation}\label{cont-eq}
\partial_t \rho + \partial_x j_1=0.
\end{equation}
Observe that the vanishing condition on $f(t,x,v)$ on the boundary implies $j_1(t,0)=j_1(t,1)=0$ and hence we deduce from \eqref{cont-eq} by integrating in $x$ that
\begin{equation}\label{charge-conserved}
\int_{\Omega} \rho(t,x)\,dx=\int_{\Omega} \rho(0,x)\, dx=\|f^0\|_{L^1( \Omega\times \R^2)}=:\|f^0\|_1.
\end{equation}
We now exploit \eqref{cont-eq} to estimate the $x$-component  $E_1$.    
By $\partial_x E_1 = \rho$ and condition \eqref{initial-E_1}, we get $E_1(t,x) = \int_0^x \rho(t,y)\, dy  + C(t)$ 
with $C(0)=\lambda = E_1(0,0)$.    Using $\partial_t E_1 = - j_1$ and \eqref{cont-eq}, we must have
\[
C'(t) 
 =- j_1(t,x) +  
 \int_0^x \partial_x j_1(t,y)\, dy=- j_1(t,0)=0.
\]
Therefore $C(t)\equiv \lambda$ and hence
\begin{equation}\label{E_1}
E_1(t,x) = \int_0^x \rho(t,y)\, dy +\lambda= \int_0^x \int_{\R^2} f(t,y,v)\, dv dy +\lambda.
\end{equation}
We conclude from \eqref{E_1} and  \eqref{charge-conserved} that 
\begin{equation}\label{bound-E_1}
\|E_1\|_{L^\infty([0,T]\times \Omega)}\leq  \|f^0\|_1 +\lambda.
\end{equation}

\subsection{Estimate  of $E_2$ and $B$}
 Let $t\in (0,T]$ and $x\in \Omega$ be fixed.    Without loss of generality, by symmetry 
 we can assume $x\leq 1/2$ in the following calculations.   
 In order to estimate $E_2$ and $B$ at the point $(t,x)$, our first step is to express these quantities  in terms of the initial and boundary data, and the current density $j_2$.   For this purpose, note that the Maxwell equations 
 \eqref{ME} yield
 \begin{equation}\label{E_2+B}
 \partial_t (E_2 +B) +\partial_x (E_2 +B)=-\partial_x B - j_2 -\partial_x E_2 +\partial_x (E_2 +B)=- j_2
 \end{equation}
 and 
 \begin{equation}\label{E_2-B}
 \partial_t (E_2 -B) -\partial_x (E_2 -B)=-\partial_x B - j_2 +\partial_x E_2 -\partial_x (E_2 -B)=- j_2.
 \end{equation}
 We now consider the following three possibilities, which depend on the relation between $x$ and $t$.  
 
 {\bf Case 1:} $0<t\leq x$. Then $0\leq x-t$ and $x+t\leq 2x\leq 1$. Therefore, it follows from \eqref{E_2+B} and \eqref{E_2-B} that
 \begin{align}
 (E_2 +B)(t,x) &= (E_2 +B)(0,x-t) -\int_0^t j_2(\tau, x-t+\tau)\, d\tau,\label{I-case1}\\
 (E_2 -B)(t,x) &= (E_2 -B)(0,x+t) -\int_0^t j_2(\tau, x+t-\tau)\, d\tau.\label{II-case1}
 \end{align}
 Adding and subtracting the two quantities respectively  yield
 \begin{align*}
 E_2(t,x) &=\frac12 \Big[E_2^0(x-t) + E_2^0(x+t) +B^0(x-t) -B^0(x+t) \Big]\\
 & -\frac 12 \int_0^t \Big[ j_2(\tau, x-t+\tau)+j_2(\tau, x+t-\tau)\Big]\, d\tau
 \end{align*}
and
 \begin{align*}
 B(t,x) &=\frac12 \Big[E_2^0(x-t) - E_2^0(x+t) +B^0(x-t) +B^0(x+t) \Big]\\
 & -\frac 12\int_0^t \Big[ j_2(\tau, x-t+\tau)-j_2(\tau, x+t-\tau)\Big]\, d\tau.
 \end{align*}
 
 {\bf Case 2:} $x<t\leq 1-x$. Then $x-t<0$ and $x+t\leq 1$. In this case \eqref{II-case1} 
 is still true, but \eqref{I-case1} is replaced by 
  \begin{equation}\label{I-case2}
 (E_2 +B)(t,x) = (E_2 +B)(t-x,0) -\int_{t-x}^t j_2(\tau, x-t+\tau)\, d\tau.
 \end{equation}
 Therefore, as in {\bf Case 1} we obtain from \eqref{I-case2} and \eqref{II-case1} that
 \begin{align*}
 E_2(t,x) &=\frac12 \Big[E_2^b(t-x,0) + E_2^0(x+t) +B^b(t-x,0) -B^0(x+t) \Big]\\
 & -\frac 12\Big[ \int_{t-x}^t  j_2(\tau, x-t+\tau) \, d\tau +\int_0^t j_2(\tau, x+t-\tau)\, d\tau\Big]
 \end{align*}
 and
 \begin{align*}
 B(t,x) &=\frac12 \Big[E_2^b(t-x,0) - E_2^0(x+t) +B^b(t-x,0) +B^0(x+t) \Big]\\
 & -\frac 12 \Big[ \int_{t-x}^t  j_2(\tau, x-t+\tau) \, d\tau -\int_0^t j_2(\tau, x+t-\tau)\, d\tau\Big].
 \end{align*}

 {\bf Case 3:} $t> 1- x$. Then $x-t<0$ and $x+t> 1$. Hence, we have \eqref{I-case2} and
  \begin{equation*}
 (E_2 -B)(t,x) = (E_2 -B)(t-1+x,1) -\int_{t-1+x}^t j_2(\tau, x+t-\tau)\, d\tau.
 \end{equation*}
 Consequently,
 \begin{align*}
 E_2(t,x) &=\frac12 \Big[E_2^b(t-x,0) + E_2^b(t-1+x,1) +B^b(t-x,0) -B^b(t-1+x,1) \Big]\\
 & -\frac12 \Big[ \int_{t-x}^t  j_2(\tau, x-t+\tau) \, d\tau +\int_{t-1+x}^t j_2(\tau, x+t-\tau)\, d\tau\Big]
 \end{align*}
 and
 \begin{align*}
 B(t,x) &=\frac12 \Big[E_2^b(t-x,0) - E_2^b(t-1+x,1) +B^b(t-x,0) +B^b(t-1+x,1) \Big]\\
 & -\frac12 \Big[ \int_{t-x}^t  j_2(\tau, x-t+\tau) \, d\tau -\int_{t-1+x}^t j_2(\tau, x+t-\tau)\, d\tau\Big].
 \end{align*}
 
 We summarize all three cases as follows.  
 \begin{lemma}\label{representation-E_2-B}
 For any $t\in (0,T]$ and $0<x\leq 1/2$, we have 
\begin{align*}
 E_2(t,x) = \frac12\big[A^+(x-t) + A^-(x+t)\big] -\frac12 \Big[ \int_{t^+(x)}^t  j_2(\tau, x-t+\tau) \, d\tau +\int_{t^-(x)}^t j_2(\tau, x+t-\tau)\, d\tau\Big]
 \end{align*}
 and
 \begin{align*}
 B(t,x) =\frac12\big[A^+(x-t) - A^-(x+t)\big] -\frac12 \Big[ \int_{t^+(x)}^t  j_2(\tau, x-t+\tau) \, d\tau -\int_{t^-(x)}^t j_2(\tau, x+t-\tau)\, d\tau\Big].
 \end{align*}
Here $A^\pm$ are given explicitly in terms of the initial and boundary data, and
\begin{equation*}
t^+(x) := \left\{\begin{array}{lr}
0 &\mbox{if}\quad  t\leq x,\\
 t-x  &\mbox{if} \quad t> x
\end{array}\right.
\qquad \mbox{and}\qquad
t^-(x) := \left\{\begin{array}{lr}
0 &\mbox{if}\quad  t\leq 1-x,\\
 t-1+x  &\mbox{if} \quad t> 1-x.
\end{array}\right.
\end{equation*}
 \end{lemma}
 
 Note that $0\leq t^-(x)\leq t^+(x)< t$ because $0\le x\le\frac12$.  
 In Case 1, $t_-(x)=t^+(x)=0$.    In Case 2, $t^-(x)=0$ but $t^+(x)\ne0$.  In Case 3, neither one is zero.  
In order to bound $E_2$ and $B$, the remaining  step is to bound the time integrals of $j_2$.  
This is accomplished  thanks to the following  variation of the cone  estimate in \cite[Lemma~1]{GSc}.

\begin{lemma}[Key cone estimate] \label{cone-estimate}
Let $t\in (0,T]$ and $x\in (0,1/2]$. Then we have
 \begin{align*}
 &\int_{t^+(x)}^{t} \int_{\R^2} |\hat v_2| f(\tau, x -t+ \tau, v) \, dv d\tau
+\int_{t^-(x)}^{t} \int_{\R^2} |\hat v_2| f(\tau, x +t- \tau, v) \, dv d\tau  \\
&\qquad\leq \int_{\Omega} e(t^-(x),y) \,dy+\int_{t^-(x)}^{t^+(x)}{E_2^b(\tau,0) B^b(\tau,0)\, d\tau},
\end{align*}
 where
 \begin{equation}\label{energy-integrand}
e(\tau,y) := \frac12 \Big[|E(\tau,y)|^2 + B(\tau,y)^2\Big] +  \int_{\R^2}{\sqrt{1+ |v|^2} f(\tau,y,v)\, dv}. 
\end{equation}
\end{lemma}
\begin{proof}
Let 
\begin{equation*}
m(\tau,y) := -  \int_{\R^2}{v_1 f(\tau,y,v)\, dv} - E_2(\tau,y) B(\tau,y).
\end{equation*}
Then by a direct calculation using \eqref{ME} and the definition of $j$, we obtain
\begin{align*}
\partial_t e -\partial_x m
&= \int_{\R^2}{\sqrt{1+ |v|^2} \, \partial_t f(\tau,y,v)\, dv} + \int_{\R^2}{v_1  \partial_x f(\tau,y,v)\, dv}\\
&\qquad \qquad -  \int_{\R^2}{(\hat v_1 E_1  +\hat v_2 E_2) f(\tau,y,v)\, dv}.
\end{align*}
Thus, it follows from the Vlasov equation \eqref{VE'} and an integration by part in $v$ that
\begin{align*}
\partial_t e -\partial_x m
&=  \int_{\R^2}{\Big[ \big(\nabla_v \sqrt{1+ |v|^2}\,\big) \cdot K  - \hat v \cdot E \Big]f(\tau,y,v)\, dv}.
\end{align*}
Since $\nabla_v \sqrt{1+ |v|^2}  =\hat v$ and $\hat v \cdot K =\hat v \cdot E$, we deduce that
\begin{equation}\label{energy-identity}
\partial_t e -\partial_x m=0\quad \mbox{in}\quad [0,T]\times \overline{\Omega}.
\end{equation}
Let us now consider the polygonal region $\Delta:= \Delta_1 \cup \Delta_2$,  
where 
\[
\Delta_1 := \Big\{(\tau, y): t^+(x) \leq \tau\leq t\, \mbox{ and }\, |y-x|\leq t-\tau \Big\}
\]
is a triangular region and
\[
\Delta_2 := \Big\{(\tau, y): t^-(x) \leq \tau\leq t^+(x)\, \mbox{ and } \, 0\leq y \leq x+t-\tau \Big\}  
\]
is a trapezoidal region.  
 We integrate the energy identity \eqref{energy-identity} over $\Delta$ and apply Green's theorem to get
 \begin{align*}
0&= \ointctrclockwise_{\partial \Delta}{\Big(m \, dt + e \,dx\Big)}
=\int_{t^-(x)}^{t}  (m-e)(\tau, x+t -\tau)\,d\tau\\
&+\int_{t}^{t^+(x)}  (m+e)(\tau, x-t +\tau)\, d\tau
+\int_{t^+(x)}^{t^-(x)}{m(\tau,0)\, d\tau}+ \int_{0}^{1} e(t^-(x),y) \,dy. 
\end{align*}
The first two terms on the right are line integrals on characteristic edges, the third one is an integral 
on the left edge where $x=0$, and the last one is an integral on the bottom edge of $\Delta$.  
Moreover, 
$m(\tau,0)= -E_2^b(\tau,0) B^b(\tau,0)$ due to  the boundary conditions for $f$ and the field.   
It follows by moving some terms around  that
\begin{align}\label{cone-est-1}
&\int_{t^+(x)}^{t}  (e+m)(\tau, x-t+\tau)\, d\tau + 
\int_{t^-(x)}^t (e-m)(\tau, x+t -\tau)\, d\tau\\
&=\int_{\Omega} e(t^-(x),y) \,dy
+\int_{t^-(x)}^{t^+(x)}{E_2^b(\tau,0) B^b(\tau,0)\, d\tau}.\nonumber
\end{align}
Notice that
\begin{align*}
e\pm m = \frac{E_1^2}{2}  + \frac{(E_2 \mp B)^2}{2} 
+  \int_{\R^2}{\big(\sqrt{1+ |v|^2} \mp v_1\big) f\, dv}\geq  \int_{\R^2}\frac{|v_2|}{\sqrt{1+ |v|^2}}  f\, dv.
\end{align*}
Therefore, we infer from \eqref{cone-est-1} that
\begin{align*}
 &\int_{t^+(x)}^{t} \int_{\R^2} |\hat v_2| f(\tau, x -t+ \tau, v) \, dv d\tau
+\int_{t^-(x)}^{t} \int_{\R^2} |\hat v_2| f(\tau, x +t- \tau, v) \, dv d\tau  \\
&\leq \int_{\Omega} e(t^-(x),y) \,dy+\int_{t^-(x)}^{t^+(x)}{E_2^b(\tau,0) B^b(\tau,0)\, d\tau}.
\end{align*}
\end{proof}
The next lemma 
states the conservation of energy.
\begin{lemma}\label{energy-conserved} Let $e(\tau,y)$ be given by \eqref{energy-integrand}. Then 
\[
\int_{\Omega}{e(t,y)\, dy}= \int_{\Omega}{e(0,y)\, dy}+\int_0^t \Big[(E_2^b B^b)(\tau,0)- (E_2^b B^b)(\tau,1)\Big] \, d\tau\quad \mbox{for all}\quad t\in [0,T].
\] 
\end{lemma}
\begin{proof}  
By the identity \eqref{energy-identity} and the boundary condition \eqref{boundary-condition}, we have
\begin{align*}
\partial_t \int_{\Omega}{e(\tau,y)\, dy} &=\int_{\Omega}{\partial_t e(\tau,y)\, dy}=\int_{\Omega}{\partial_x m(\tau,y)\, dy}=m(\tau,1) - m(\tau,0)\\
&=  \int_{\R^2}{v_1 \Big[ f(\tau,0,v) - f(\tau,1,v)\Big]\, dv} + (E_2 B)(\tau,0)- (E_2 B)(\tau,1)\\
&=(E_2^b B^b)(\tau,0)- (E_2^b B^b)(\tau,1).
\end{align*}
The lemma follows by integration. 
\end{proof}
We now combine the preceding results.  
\begin{corollary}\label{cor:Bound-Fields} The field is bounded as follows: 
$\|E_1\|_{L^\infty([0,T]\times \Omega)}\leq \|f^0\|_1 +\lambda$, and 
\begin{equation}\label{bound-E_2-B_2}
\|E_2\|_{L^\infty([0,T]\times \Omega)}, \, \|B\|_{L^\infty([0,T]\times \Omega)}\leq C_1, 
\end{equation}
where $C_1 :=\|E_2^0\|_{L^\infty(\Omega)} +\|E_2^b\|_{L^\infty([0,T]\times \partial\Omega)} + \|B^0\|_{L^\infty(\Omega)} +
\|B^b\|_{L^\infty([0,T]\times \partial\Omega)}+\frac14 \big[ (\|f^0\|_1 +\lambda)^2 +\|E_2^0\|_{L^\infty(\Omega)}^2 +\|B^0\|_{L^\infty(\Omega)}^2+4 T \|E_2^b B^b\|_{L^\infty([0,T]\times \partial\Omega)} \big] +\frac12 \|\sqrt{1+|v|^2} f^0\|_1$. 
\end{corollary}
\begin{proof}
The estimate for $E_1$ is from \eqref{bound-E_1} and we only need to prove \eqref{bound-E_2-B_2}. Let $t\in (0,T]$ and $x\in\Omega$.  
By symmetry we can assume  $x\leq 1/2$ as the case $x>1/2$ is similar. 
By Lemma~\ref{representation-E_2-B} and the explicit formulas for $A^\pm$ given in the three cases considered above, we have 
\begin{align*}
 |E_2(t,x)|,\, |B_2(t,x)| &\leq 
\|E_2^0\|_{L^\infty(\Omega)} +\|E_2^b\|_{L^\infty([0,T]\times \partial\Omega)} + \|B^0\|_{L^\infty(\Omega)} +
\|B^b\|_{L^\infty([0,T]\times \partial\Omega)}\\
 & +\frac12 \Big[ \int_{t^+(x)}^t  |j_2|(\tau, x-t+\tau) \, d\tau +\int_{t^-(x)}^t |j_2|(\tau, x+t-\tau)\, d\tau\Big].\nonumber
 \end{align*}
  But it follows from Lemmas~\ref{cone-estimate} and \ref{energy-conserved} that  
 \begin{align*}
 &\int_{t^+(x)}^t  |j_2|(\tau, x-t+\tau) \, d\tau +\int_{t^-(x)}^t |j_2|(\tau, x+t-\tau)\, d\tau\\
 &\leq \int_{\Omega} e(0,y) \,dy + \int_{0}^{t^+(x)}{E_2^b(\tau,0) B^b(\tau,0)\, d\tau}- \int_{0}^{t^-(x)}{E_2^b(\tau,1) B^b(\tau,1)\, d\tau}
   \\
 &\leq \frac{1}{2}  \Big[ ( \|f^0\|_1 +\lambda)^2 +\|E_2^0\|_{L^\infty(\Omega)}^2 +\|B^0\|_{L^\infty(\Omega)}^2\Big] + \|\sqrt{1+|v|^2} f^0\|_{L^1(\Omega\times \R^2)}
+  2T \|E_2^b B^b\|_{L^\infty([0,T]\times \partial\Omega)}. 
 \end{align*}
 Therefore we obtain the desired estimate \eqref{bound-E_2-B_2}.
\end{proof}

\section{Confinement of the particles}\label{sec:confinement}
Given $(t,x,v)\in (0,T]\times \Omega\times \R^2$. The  characteristics of \eqref{VE} corresponding to the point $(t,x,v)$ are
the solutions $s\mapsto \big(X(s), V(s)\big)= \big(X(s;t,x,v), V(s;t,x,v)\big)$ to the system 
\begin{equation}\label{CE}
\left\{\begin{array}{rl}
&\frac{dX}{ds} = \hat V_1(s), \\
&\frac{d V_1}{d s} = E_1(s, X) +\hat V_2(s)\, \check B(s, X),\\
&\frac{d V_2}{d s} = E_2(s, X) -\hat V_1(s)\, \check B(s, X),\\
&X(t;t, x,v) = x,\quad V(t;t,x,v) = v.  
\end{array}\right.
\end{equation}
Assuming that $E_1, E_2, \check B \in C^1([0,T]\times \Omega)$,  
there exists a unique $C^1$ solution $(X, V)$ to the system \eqref{CE} in some time interval.   
It can be uniquely extended to the whole time interval $[0,T]$ as long as the solution $X(s)$ does not reach the boundary 
$\partial\Omega$. In the next lemma, we show that this is indeed the case thanks to  condition \eqref{blow-up-cond} for the potential of the external magnetic field.

\begin{lemma}[Confinement property]\label{lm:confi} 
Assume that  $E, B \in C^1([0,T]\times \Omega)$ satisfy $\partial_t B=-\partial_x E_2$,  and that 
there exist  constants $C_0,\, C_0'>0$  such that 
\begin{equation}\label{field-condition}
|E(s,y)|\leq C_0\quad \mbox{and}\quad |B(s,y)|\leq  C_0'\quad \mbox{for all }\, (s,y)\in [0,T]\times \Omega.
\end{equation}
Let  $(t,x,v)\in (0,T]\times \Omega\times \R^2$ and $(X(s), V(s))$  be a $C^1$ solution to \eqref{CE} 
in the time interval  $[t-\alpha, t+\alpha]$ for some  $\alpha>0$. Then
\begin{equation}\label{X-away-boundary}
\dist(X(s),\partial \Omega)^\gamma \geq \frac{c_0}{c_0^{-1} + 2|v| +  C_0'  + 3 C_0\alpha +|\psie(x)|} \quad \forall s\in [t-\alpha,t+\alpha].
\end{equation}
\end{lemma}
\begin{proof}
By assumption,   $X(s)\in \Omega$ for every $s\in (t-\alpha,t+\alpha)$. Let us consider the case when $s\le t$; the other being similar.  
Since
\begin{align*}
\frac{d}{d s} |V|^2
&= 2\Big[V_1 \dot V_1 + V_2 \dot V_2 \Big]
=2\Big[V_1 E_1(s, X) + V_1 \hat V_2 \check B(s,X) + V_2E_2(s, X) - V_2 \hat V_1 \check B(s, X) \Big]\\
&= 2 V\cdot E(s, X),
\end{align*}
we deduce  from the bound in \eqref{field-condition} that
\begin{align*}
|V(s)|^2
\leq |v|^2  + 2 C_0 \int_s^t{|V(\tau)| \, d\tau} \quad \mbox{for}\quad s\in [t-\alpha,t]. 
\end{align*}
Hence $u(s) := \sup_{s\leq \tau\leq t}{|V(\tau)|}$  satisfies   
\[
u(s)^2 \leq |v|^2 + 2 C_0\alpha \,u(s). 
\]
It follows that $ u(s)\leq |v| + 2 C_0 \alpha$  and so
\begin{equation}\label{support-est}
|V(s)| \leq  |v| + 2 C_0 \alpha\quad \forall s\in [t-\alpha,t].
\end{equation}

To estimate $X(s)$,  let $\psi(\tau,y) := \int_{\frac{1}{2}}^y{B(\tau,z)\, dz}$.    
Then thanks to $\partial_t B=-\partial_x E_2$, we get
\begin{equation}\label{time-deri-potential}
\partial_t \psi(\tau,y) = -\int_{\frac{1}{2}}^y{\partial_x E_2(\tau,z)\, dz}
= E_2 \Big(\tau, \frac{1}{2}\Big) - E_2(\tau,y).
\end{equation}
Next define
\[
p(\tau,y,w) := w_2 +\psi(\tau,y) +\psie(y) 
\] 
where $w=(w_1,w_2)\in \mathbb{R}^2$. Differentiating $p(\tau,y,w)$ along the characteristics and using \eqref{CE} and \eqref{time-deri-potential}, we obtain
\begin{align*}
\frac{d}{ds} p\big(s, X(s), V(s)\big)
&=\dot V_2 + \partial_t \psi(s, X) + \dot X \,\partial_x \psi(s, X)  + \dot X \, \partial_x \psie(X) \\
&=E_2(s, X) -\hat V_1 [B(s,X) +\Be(X)] + \partial_t \psi(s, X) + \hat V_1 B(s, X)  + \hat V_1 \Be(X) 
\\&=E_2\Big(s,\frac{1}{2}\Big).
\end{align*} 
Therefore 
\begin{equation}\label{Charac-identity}
V_2(s) +\psi(s,X(s)) +\psie(X(s))
=v_2 +\psi(t,x) +\psie(x)-\int_s^t{E_2\Big(\tau, \frac{1}{2}\Big)\, d\tau}
\end{equation}
for every $s\in [t-\alpha, t]$.

 We now show using \eqref{blow-up-cond} that the path $\tau\in [t-\alpha, t]\mapsto X(\tau)$ stays 
away from $\partial\Omega$ by a specific distance depending on $x$ and $v$.  
For this purpose, let $\tau_0 \in (t-\alpha,t)$ be arbitrary.    
 Two of the terms  in \eqref{Charac-identity} are bounded as 
 \[
 | \psi(\tau,y) |  \le  \Big| \int_{\frac{1}{2}}^y{B(\tau,z)\, dz} \Big|  \le  \frac{C_0'}{2}.  
 \]
By  \eqref{support-est},  $|V(\tau_0)|\leq |v| + 2 C_0 \alpha$.
We deduce from \eqref{Charac-identity}  that
\begin{align*}
\big|\psie(X(\tau_0))-\psie(x)\big|
\leq 2 |v| +  C_0'  + 3 C_0\alpha. 
\end{align*}
This together with the   assumption in \eqref{blow-up-cond}  implies that 
\[
\dist(X(\tau_0),\partial \Omega)^\gamma \geq \frac{c_0}{c_0^{-1} + 2|v| +  C_0'  + 3 C_0\alpha +|\psie(x)|}.
\]
\end{proof}

\begin{remark}\label{rm:cond-confine}
The condition $\partial_t B =-\partial_x E_2$ is not necessary for the validity of Lemma~\ref{lm:confi}. Indeed, an inspection of the above proof reveals
 that it is enough to assume  the quantity $\partial_t B +\partial_x E_2$ to be bounded.
\end{remark}

Lemma~\ref{lm:confi} shows that the particles never reach $\partial\Omega$ in a finite time. 
As a consequence, we obtain the following corollary.  
\begin{corollary}\label{cor:ODEs}
 Let   $E$ and $B$ be as in Lemma~\ref{lm:confi}. Then for any  $(t,x,v)\in (0,T]\times \Omega\times \R^2$,
 the characteristic system \eqref{CE} admits a unique $C^1$ solution $(X(s), V(s))$  in $[0,T]$ with $X(s)\in \Omega$ 
 \,for every $s\in [0,T]$. 
 \end{corollary}

We end this section by giving some direct consequences of Corollary~\ref{cor:Bound-Fields} and Corollary~\ref{cor:ODEs} 
which will be needed in what follows. We still suppose  $(f, E_1, E_2, B)$ is a $C^1$ solution as in
Section~\ref{sec:bound-fields}. 
 Then thanks to Corollary~\ref{cor:Bound-Fields}, 
 the conclusion about the characteristics in Corollary~\ref{cor:ODEs} is true. 
 Since the solution $f$ to \eqref{VE'} is constant along such characteristics, we have
\begin{equation}\label{formula-f}
f(t,x,v) =f^0\big(X(0; t,x,v), V(0;t,x,v)\big).
\end{equation}
It follows that 
\begin{equation}\label{bound-f}
\|f\|_{L^\infty([0,T]\times \Omega\times \R^2)}=\|f^0\|_{L^\infty( \Omega\times \R^2)}=:\|f^0\|_\infty.
\end{equation}
The next result shows that $f(t,\cdot,\cdot)$ has compact support 
in both $x$ and $v$ variables. 
\begin{lemma}\label{lm:-support}
Define
\begin{align*}
P(t) &:= \sup{\Big\{|v|:\, f(t,x,v)\neq 0 \quad\mbox{for some}\quad x\in \Omega\Big\}},\\
\Sigma(t)  &:= \Big\{x\in \Omega:\, f(t,x,v)\neq 0 \quad\mbox{for some}\quad v\in \R^2\Big\}.
\end{align*}
If  $\spt(f^0)\subset [\epsilon_0, 1-\epsilon_0]\times \{ |v|\le k_0\}\,$ for some $\epsilon_0,\, k_0>0$,   
then we have
\begin{align}
P(t) &\leq 
 k_0 + C_2 t,\label{v-spt}\\
\dist\big(\Sigma(t),\partial\Omega\big)^\gamma
&\geq  \frac{c_0}{c_0^{-1} + 2k_0+ C_1 + 3 C_2 t+ \|\psie\|_{L^\infty([\epsilon_0, 1-\epsilon_0])}},\label{x-spt}
\end{align}
 for  $t\in [0,T]$, 
where $C_1$ is  given by Corollary~\ref{cor:Bound-Fields} and $C_2 :=4(\|f^0\|_1+\lambda + C_1)$. 
\end{lemma}
\begin{proof}
Let $(t,x,v)\in (0,T]\times \Omega\times \R^2$ and consider the corresponding characteristic curve  
$(X(s), V(s))$
given by  \eqref{CE}. Since  $\|E\|_{L^\infty([0,T]\times\Omega)}
\leq  C_0 :=\|f^0\|_1 +\lambda +C_1$ by Corollary~\ref{cor:Bound-Fields},  we have as in the proof of Lemma~\ref{lm:confi} that
\begin{equation}\label{control-V}
|V(s)| \leq |v| + 2 C_0 t\quad \forall s\in [0,t].
\end{equation}
Using  the fact that 
$\frac{d}{d s} |V|^2   =   2 V\cdot E(s, X)  $, 
we obtain 
\begin{align*}
|v|^2&=|V(0)|^2 + 2\int_0^t V(s)\cdot E(s, X)\, ds
\leq |V(0)|^2 + 2C_0 t \, \big(|v| + 2C_0 t\big). 
\end{align*}
It follows that $|v|\leq |V(0)| + 4 C_0 t$.   
Together with \eqref{formula-f} and the support assumption on $f^0$, this yields 
\[ P(t) \leq 
k_0 +4 C_0 t.
\]
It remains to show \eqref{x-spt}. Fix $t\in (0,T]$ and let $x\in \Sigma(t)$.   
Then there exists $v\in \R^2$ such that $f(t,x,v)\neq 0$. By \eqref{v-spt}, we have $|v|\leq k_0 + C_2 t$.   
Also, it  follows from the formula \eqref{formula-f} for  $f(t,x,v)$  and the assumption on $f^0$ that the 
corresponding characteristics $(X(s), V(s))$ must satisfy $X(0) \in [\epsilon_0, 1-\epsilon_0]$.   
Moreover, the identity \eqref{Charac-identity} is valid for all $s\in [0,t]$, which yields  in particular
\begin{equation*}
V_2(0) +\psi(0,X(0)) +\psie(X(0))
=v_2 +\psi(t,x) +\psie(x)-\int_0^t{E_2(\tau, \frac{1}{2})\, d\tau}.
\end{equation*}
Using  Corollary~\ref{cor:Bound-Fields} and \eqref{control-V}, we deduce  that
\begin{align*}
|\psie(x)|
&\leq 2k_0 +C_1 + (2 C_2+ 2 C_0  + C_1) t +\|\psie\|_{L^\infty([\epsilon_0, 1-\epsilon_0])}\\
&\leq 2k_0 +C_1 + 3 C_2 t +\|\psie\|_{L^\infty([\epsilon_0, 1-\epsilon_0])}=:C.
\end{align*}
We infer from this and \eqref{blow-up-cond} that 
$\dist(x,\partial\Omega)^\gamma
\geq  c_0/(c_0^{-1} +C)$ and \eqref{x-spt} follows.
\end{proof}

By Lemma~\ref{lm:-support} and \eqref{bound-f}, we  have
\[
\int_{\R^2} f(t,x,v)\, dv = \int_{|v|\leq k_0+ C_2 t} f(t,x,v)\, dv\leq  \|f^0\|_\infty (k_0+ C_2 t)^2.
\]
This immediately  leads to: 
\begin{corollary}\label{cor:rho-j-est} We have
\[\|\rho\|_{L^\infty([0,T]\times \Omega)}, \, \|j\|_{L^\infty([0,T]\times \Omega)}
\leq\|f^0\|_\infty (k_0+ C_2 T)^2.
\]
\end{corollary}

\section{Bounds on derivatives of the fields}\label{sec:bound-derivatives-fields}
In this section we first derive $L^\infty$ estimates for derivatives of the fields and then use them to obtain similar estimates for derivatives of the distribution function $f$.

Let $k^\pm(t,x) := (E_2 \pm B)(t,x)$. By the arguments leading to Lemma~\ref{representation-E_2-B}, we have for every $t\in (0,T]$ and $0<x\leq 1/2$ that
\begin{align}\label{k-rep-formula}
 k^+(t,x) &= \frac12 A^+(x- t)  -  \int_{t^+(x)}^t  j_2(\tau, x- t+ \tau) \, d\tau,\\
  k^-(t,x) &= \frac12 A^-(x-t)  -  \int_{t^-(x)}^t  j_2(\tau, x+ t- \tau) \, d\tau, \nonumber
 \end{align}
 in which $A^\pm$ are expressed in terms of the initial and boundary data; see Lemma \ref{representation-E_2-B}. 
  These representation formulas
 play an important role in the proof of the next result. Before stating it, let  $\theta_0$ and $\theta_1$ denote  the small constants given  by 
\begin{align}
 \theta_0^\gamma &:=\frac{c_0}{c_0^{-1} + 2k_0+ C_1 + 3 C_2 T+ \|\psie\|_{L^\infty([\epsilon_0, 1-\epsilon_0])}},\label{eq:theta-0}\\
\theta_1^\gamma &:=\frac{c_0}{c_0^{-1} + 2k_0+ C_1 + 3 C_2 T+ \|\psie\|_{L^\infty([\theta_0, 1-\theta_0])}}\label{eq:theta-1},
\end{align}
in which $\epsilon_0$ is defined as in Lemma~\ref{lm:-support}. Notice that the choice of $\theta_0$ ensures that the $x$-support of $f(t)$ is contained in $[\theta_0, 1-\theta_0]$ for every $t\in [0,T]$ (see Lemma~\ref{lm:-support}). On the other hand, Corollary~\ref{cor:Bound-Fields} and Lemma~\ref{lm:confi} imply that the characteristics $(X(s), V(s))$ corresponding to any point 
$(t,x,v)\in (0,T]\times [\theta_0, 1-\theta_0] \times \bar B_{k_0 + C_2 T} $ satisfy: $X(s)\in [\theta_1, 1-\theta_1]$ for all $s\in [0,t]$.
 
\begin{lemma}\label{bound-deri-fields} There exists a constant $C_T>0$ depending only on $k_0,\, T,\,\lambda,\, \|f^0\|_\infty$, $\|\Be\|_{L^\infty([\theta_0, 1-\theta_0])}$, the $C^1$ norms of $E_2^0,\, 
 B^0$ on $\Omega$,  and the $C^1$ norms of $ 
E_2^b(\cdot,x),\, B^b(\cdot,x)$ on $[0,T]$ ($x=0,\, 1$) such that
\[
\|\partial_x k^\pm\|_{L^\infty([0,T]\times \Omega)}\leq C_T.
\]
Consequently, we have $\|\partial_x E_2\|_{L^\infty([0,T]\times \Omega)}, \, \|\partial_x B\|_{L^\infty([0,T]\times \Omega)}\leq C_T$.
\end{lemma}
\begin{proof}
We  employ an argument similar to the proof of \cite[Lemma~3]{GSc}. 
For simplicity, we derive the $L^\infty$ estimates in the region $[0,T]\times (0, 1/2]$ as the case $x>1/2$ is similar. 
For $(t,x)$ in such region, it follows from \eqref{k-rep-formula} by  differentiating $k^+$ in $x$ that
\begin{align}\label{k_x}
 \partial_x k^+(t,x)
 &=M(t,x)-  \int_{t^+(x)}^t  \partial_x j_2(\tau, x- t+ \tau) \, d\tau\\
 &=M(t,x)-  \int_{t^+(x)}^t  \int_{\R^2} \hat v_2 \, \partial_x f(\tau, x- t+ \tau,v) \, dv d\tau\nonumber 
  \end{align}
with $M(t,x) :=\frac12 (A^+)'(x- t) + j_2\big(t^+(x), x-t +t^+(x)\big)\, (t^+)'(x)$.  Notice that by using the explicit formula for $A^+$,  Corollary~\ref{cor:rho-j-est} and the fact $|(t^+)'(x)|\leq 1$, we obtain 
\begin{equation}\label{bound-M}
\|M\|_{L^\infty([0,T]\times \Omega)}\leq C,
\end{equation}
where $C$ depends only on $k_0,\, T,\,\lambda,\, \|f^0\|_\infty$, the $L^\infty$ norms of the derivatives of $E_2^0$, 
 $B^0$ on $\Omega$, and  the $L^\infty$ norms of the derivatives of $E_2^b(\cdot,x)$, $B^b(\cdot, x)$ on $[0,T]$ ($x=0,\, 1$).

We next use the splitting method of 
Glassey and Strauss in \cite{GStr} and  \cite{GSc} to express 
the operator $\partial_x$ in terms of the two differential operators 
\[
T_+ =\partial_t +\partial_x \quad\mbox{and}\quad S =\partial_t +\hat v_1 \partial_x.
\]
Obviously 
\begin{equation}\label{GS-idea}
\partial_x = \frac{T_+-S}{1-\hat v_1},
\end{equation}
so that \eqref{k_x} can be written as
 \begin{align*}
 \partial_x k^+(t,x)
 &=M(t,x)-  \int_{t^+(x)}^t  \int_{\R^2} \frac{\hat v_2}{1-\hat v_1} \big[ (T_+ f)(\tau, x- t+ \tau,v) -(S f)(\tau, x- t+ \tau,v)\big] \, dv d\tau\\
 &=M(t,x)-  \int_{t^+(x)}^t  \frac{d}{d\tau}\int_{\R^2} \frac{\hat v_2}{1-\hat v_1} f(\tau, x- t+ \tau,v)  \, dv \,\, d\tau\\
 & \qquad \qquad -  \int_{t^+(x)}^t  \int_{\R^2} \frac{\hat v_2}{1-\hat v_1} \nabla_v \cdot (Kf)(\tau, x- t+ \tau,v) \, dv d\tau,
  \end{align*}
  where  we have used the Vlasov equation $Sf +\nabla_v \cdot (Kf) =0$.
  Since $f$ has compact support in $v$ by Lemma~\ref{lm:-support}, we easily integrate the last term by parts to arrive at 
  the equation 
  \begin{align*}
 \partial_x k^+(t,x)
 &=M(t,x)-  \int_{\R^2} \frac{\hat v_2}{1-\hat v_1} f(t, x,v)  \, dv + \int_{\R^2} \frac{\hat v_2}{1-\hat v_1} f\big(t^+(x), x- t+ t^+(x),v\big)  \, dv\\
 & \qquad \qquad +  \int_{t^+(x)}^t  \int_{\R^2} \nabla_v \Big(\frac{\hat v_2}{1-\hat v_1} \Big) \cdot (Kf)\big(\tau, x- t+ \tau,v\big) \, dv d\tau.
  \end{align*}
 We know the support of $f$ in $v$ is contained in the ball $\overline{B_R}$, where $R:= k_0+ C_2 T$ with $C_2$ being given in  
 Lemma~\ref{lm:-support}.
 Using this together with  \eqref{bound-M} and \eqref{bound-f}, we deduce  that
  \begin{align}\label{eq:k_x-inter}
 \|\partial_x k^+\|_{L^\infty([0,T]\times(0,\frac12])}
 &\leq C +  2\pi R^2 \|f^0\|_\infty \Big \|\frac{\hat v_2}{1-\hat v_1}\Big\|_{L^\infty(B_R)}\\
 & +  \Big \| \nabla_v \Big(\frac{\hat v_2}{1-\hat v_1} \Big)\Big\|_{L^\infty(B_R)}\int_{x-t+t^+(x)}^x  \int_{B_R} (|K| f)(y-x+t, y,v) \, dv dy.\nonumber
  \end{align}
  But it follows from \eqref{x-spt} and the definition of $\theta_0$ in \eqref{eq:theta-0} that  
  \[
  \int_{x-t+t^+(x)}^x  \int_{B_R} (|K| f)(y-x+t, y,v) \, dv dy
  \leq \int_{\theta_0}^{1-\theta_0}  \int_{B_R} (|K| f)(y-x+t, y,v) \, dv dy.
  \]
  Also, 
  Corollary~\ref{cor:Bound-Fields} yields $\|K\|_{L^\infty([0,T]\times [\theta_0, 1-\theta_0])}\leq C' := C_2 + \|\Be\|_{L^\infty( [\theta_0, 1-\theta_0])}$.  Thus we obtain from \eqref{eq:k_x-inter} that
  \begin{align*}
 \|\partial_x k^+\|_{L^\infty([0,T]\times(0,\frac12])}
 &\leq C + \pi R^2 \|f^0\|_\infty  \left\{ 2\Big \|\frac{\hat v_2}{1-\hat v_1}\Big\|_{L^\infty(B_R)} + C'\,  \Big\| \nabla_v \Big(\frac{\hat v_2}{1-\hat v_1} \Big)\Big\|_{L^\infty(B_R)}\right\}\leq C_T   
  \end{align*}
  for some constant $C_T$.  
  By a similar argument for the case $t\in [0,T]$ and $x\in (1/2,1)$, we infer further that $\|\partial_x k^+\|_{L^\infty([0,T]\times\Omega)}
 \leq C_T$.
  The bound for $\partial_x k^-$ is obtained in the same manner. The only change is in place of \eqref{GS-idea} we now express
$\partial_x = \frac{S -T_{-}}{1+\hat v_1}$ with $T_{-}= \partial_t -\partial_x$. The differential operator $T_{-}$ is employed to ensure that
\[
\frac{d}{d\tau} f(\tau, x+ t- \tau,v)=(T_{-}f)(\tau, x+ t- \tau,v).
\]
\end{proof}

We next exploit the Vlasov and Maxwell equations to derive estimates for 
all the first derivatives of $E$, $B$ and $f$.
\begin{lemma}\label{C^1-est} Assume in addition that $f \in C^2([0,T]\times \Omega\times \R^2)$. 
There exists a constant $C_T>0$ depending only on $k_0,\, T, \, \lambda$, $\|\Be\|_{C^1([\theta_1, 1-\theta_1])}$, the $C^1$ norms of $f^0,\, E_2^0,\, 
 B^0$,   and the $C^1$ norms of $ 
E_2^b(\cdot,x),\, B^b(\cdot,x)$ on $[0,T]$ ($x=0,\, 1$) such that
\[
\|f\|_{C^1([0,T]\times \overline\Omega\times \R^2)} + \|E\|_{C^1([0,T]\times \overline\Omega)}+\|B\|_{C^1([0,T]\times \overline\Omega)}\leq C_T.
\]
\end{lemma}
\begin{proof}
We begin with the fields $E$ and $B$. Since $\partial_t E_1 =- j_1$ and $\partial_x E_1 = \rho$, we get from Corollary~\ref{cor:rho-j-est} that
$\|\nabla E_1\|_{L^\infty([0,T]\times \overline\Omega)}\leq 2 \|f^0\|_\infty (k_0+ C_2 T)^2$. Using $\partial_t E_2 =-\partial_x B - j_2$, Lemma~\ref{bound-deri-fields} and 
Corollary~\ref{cor:rho-j-est}, we also get an $L^\infty$ bound for 
 $\nabla E_2$. These 
 together with  Corollary~\ref{cor:Bound-Fields}  give $\| E\|_{C^1([0,T]\times\overline\Omega)}\leq C_T$. On the other hand, the $C^1$ estimate for $B$  is a consequence of the fact $\partial_t B=-\partial_x E_2$, Lemma~\ref{bound-deri-fields} and
Corollary~\ref{cor:rho-j-est}.

Next we estimate the derivatives of $f$. By differentiating the Vlasov equation \eqref{VE'} with respect to $x$ and $v$ respectively, one has
\begin{align*}
\big(\partial_t +\hat v_1 \partial_x + K\cdot\nabla_v\big)(\partial_x f)
&= -\partial_x K \cdot \nabla_v f,\\
\big(\partial_t +\hat v_1 \partial_x + K\cdot\nabla_v\big)(\nabla_v f)
&= -(\nabla_v \hat v_1) \partial_x f - (\nabla_v\cdot K)\nabla_v f.
\end{align*}
Let $R := k_0 + C_2 T$. 
Integrating the two equations along the characteristics and using the remark just before Lemma~\ref{bound-deri-fields}, we obtain
\begin{align*}
\|\partial_x f(t)\|_{L^\infty([\theta_0, 1-\theta_0]\times \bar B_{R})}\leq
\|\partial_x f^0\|_\infty +\int_0^t \|\partial_x K\|_{L^\infty([0,T]\times [\theta_1, 1-\theta_1]\times \R^2)}   \|\nabla_v f(s)\|_\infty\, ds
\end{align*}
and
 \begin{align*}
&\|\nabla_v f(t)\|_{L^\infty([\theta_0, 1-\theta_0]\times \bar B_{R})}\\ &\leq
\|\nabla_v f^0\|_\infty
+ \int_0^t \Big[\|\nabla_v \hat v_1\|_\infty \|\partial_x  f(s)\|_\infty +\|\nabla_v\cdot K\|_{L^\infty([0,T]\times [\theta_1, 1-\theta_1]\times \R^2)} \|\nabla_v f(s)\|_\infty\Big] \, ds.
\end{align*}
 Observe that $\|\nabla_v \hat v_1\|_\infty \leq 2$. Moreover, the $C^1$ bounds for $E$, $B$ and the assumption for $\Be$ imply that 
$\|\partial_x K\|_{L^\infty([0,T]\times [\theta_1, 1-\theta_1]\times \R^2)}\leq C_T$ and $\|\nabla_v\cdot K\|_{L^\infty([0,T]\times [\theta_1, 1-\theta_1]\times \R^2)} \leq C_T$, where $C_T$ now  depends also on $\|\Be\|_{C^1([\theta_1, 1-\theta_1])}$. Therefore, it follows from the above two inequalities and the fact $f(t)$ is supported 
in $[\theta_0, 1-\theta_0]\times \bar B_R$ that
\begin{align}
\|\partial_x f(t)\|_\infty 
&\leq
C\left( 1+   \int_0^t   \|\nabla_v f(s)\|_\infty\, ds\right),\label{eqf_x}\\
\|\nabla_v f(t)\|_\infty &\leq
C \left(1+  \int_0^t \big[ \|\partial_x  f(s)\|_\infty + \|\nabla_v f(s)\|_\infty\big] \, ds \right).\label{eq:f_v}
\end{align}
Letting  
$u(s) := \|\partial_x f(s)\|_\infty +\|\nabla_v f(s)\|_\infty$,      
we get
$u(t)\leq 2C \left( 1+   \int_0^t  u(s)\, ds\right),  $   
so that 
$u(t)\leq 2C e^{2Ct}\leq 2C e^{2CT}$  for $t\in [0,T],    $  
giving the bounds for $\|\partial_x f\|_\infty$ and $\|\nabla_v f\|_\infty$.   
The identity 
\[
\partial_t f = -\hat v_1 \partial_x f -K\cdot \nabla_v f,
\]
also yields the bound on $\|\partial_t f\|_\infty$. 
\end{proof}

\section{Proof of the main result}


\begin{proof}[\textbf{Proof of Theorem~\ref{thm:main}, the uniqueness part}]
Suppose  that $(\tilde  f, \tilde E, \tilde B)$ and $(f^*, E^*, B^*)$  are two global $C^1$ solutions to 
the problem \eqref{VE}--\eqref{boundary-condition}. Define
\[
f := \tilde f - f^*,\quad E :=\tilde E- E^* \quad \mbox{and}\quad B := \tilde B - B^*.
\]
Then we have 
\begin{equation}\label{eq:difference-E_1}
E_1(t,x) = \int_0^x \int_{\R^2} f(t,y,v)\, dv dy
\end{equation}
and
\begin{equation}\label{eq-for-f-f}
\partial_t f + \hat v_1 \partial_x f + \big[ E^* + (\hat v_2, -\hat v_1) (B^* +\Be)\big] \cdot \nabla_v f =- \big[ E + (\hat v_2, -\hat v_1) B\big] \cdot \nabla_v \tilde f.
\end{equation}
Let $T>0$ be arbitrary.  Lemma~\ref{lm:confi} implies that the characteristics for equation \eqref{eq-for-f-f} 
never reach  $\partial \Omega$.  
So, integrating \eqref{eq-for-f-f} along  characteristics and using $f(0,\cdot,\cdot)\equiv 0$, we obtain for every $t\in [0,T]$ that
\begin{equation}\label{eq:f-f}
\| f(t)\|_{\infty}\leq \|\nabla_v \tilde f\|_\infty\int_0^t \Big(\|E(s)\|_\infty + \|B(s)\|_\infty \Big) \, ds.
\end{equation}
The relation \eqref{eq:difference-E_1}  and Lemma~\ref{lm:-support}  yield
\begin{equation}\label{eq:E_1-E_1}
\|E_1(t)\|_\infty  \leq  \int_0^1 \int_{B_{k_0 + C_2 T}} f(t,y,v)\, dv dy\leq C_T \|f(t)\|_\infty.
\end{equation}
On the other hand, we  infer from the representation formulas for $\tilde E_2,\, \tilde B$ and $E^*_2,\, B^*$ given by Lemma~\ref{representation-E_2-B} that
\begin{equation}\label{eq:B-B}
\|E_2(t)\|_\infty, \, \|B(t)\|_\infty  \leq   \int_0^t \int_{B_{k_0 + C_2 T}} |\hat v_2|\,  \|f(\tau)\|_\infty \, dv d\tau
\leq C_T \int_0^t  \|f(s)\|_\infty \, ds.
\end{equation}
Letting 
$h(s) :=\sup_{\tau \in [0,s]}{\|f(\tau)\|_\infty}$,   
it follows from \eqref{eq:f-f}--\eqref{eq:B-B} that there exists a constant $C>0$ depending on $C_T$ and $\|\nabla_v \tilde f\|_\infty=
\|\nabla_v \tilde f\|_{L^\infty\big([0,T]\times \overline{\Omega}\times \bar B_{k_0+ C_2 T}\big)}<\infty$ such that
\[
\| f(t)\|_{\infty} \leq  C \int_0^t h(s)\, ds,\ \forall t\in [0,T].
\]
Thus $ h(t) \leq  C \int_0^t h(s)\, ds , \  \forall t\in [0,T]$,  
so that $h\equiv 0$, and hence  $f(t)=0$  for every $t\in [0,T]$. This together with \eqref{eq:E_1-E_1} and \eqref{eq:B-B}  gives also $\|E(t)\|_\infty = \|B(t)\|_\infty=0$. Therefore we conclude that 
$\tilde f(t)\equiv f^*(t)$, $\tilde E(t)  \equiv E^*(t)$ and $\tilde B(t)\equiv B^*(t)$ for all $t\in [0,T]$. The global uniqueness follows since $T>0$ is arbitrary.
\end{proof}

\begin{proof}[\textbf{Proof of Theorem~\ref{thm:main}, the existence part}]
Given our results obtained in
Sections \ref{sec:bound-fields}--\ref{sec:bound-derivatives-fields}, the proof of 
the existence of a global $C^1$ solution  follows via a  
the standard iteration scheme. This procedure  is presented in \cite{GSc} and \cite[Chapter~5]{G}, and we shall only  indicate the main points. By a 
standard density argument, one can assume  in addition that $\psie\in C^3(\Omega)$,  $f^0\in C^2_0(\Omega\times \R^2)$, $E_2^0,\, B^0 \in C^2(\overline{\Omega})$ and  $E_2^b(\cdot,x), B^b(\cdot,x)  \in C^2([0,\infty))$ at each $x=0,1$. 

Let  $T>0$ be arbitrary. We recursively define a sequence of solutions $\{(f^n, E^n, B^n)\}$ to the corresponding linear equations and show that it converges
to a solution of 
 the nonlinear problem \eqref{VE}--\eqref{boundary-condition}.  
 For the initial step ($n=0$), we take  $f^0(t,x,v) := f^0(x,v)$, and
 \[
  E^0(t,x) :=  \int_{0}^x \int_{\R^2} f^0(y,v)\, dv dy +\lambda,\quad E^0_2(t,x) :=E^0_2(x),\quad B^0(t,x) := B^0(x).
 \]
For $n\in \N$, assume that $E^{n-1}_1,\, E^{n-1}_2,\, B^{n-1}\in C^2([0,T]\times \overline{\Omega})$ are already given.
Let $K^{n-1} := E^{n-1} + (\hat v_2, -\hat v_1) \big(B^{n-1} + \Be\big)$ and denote 
$\big(X^n(s), V^n(s)\big)$  the solution of the characteristics system associated to a point $(t,x,v)\in [0,T]\times \Omega\times \R^2$. That is,
\begin{equation}\label{CE-n}
\left\{\begin{array}{rl}
&\frac{dX^n}{ds} = \hat V^n_1(s), \\
&\frac{d V^n}{d s} = K^{n-1}(s, X^n, V^n),\\
&X^n(t; t, x,v) = x,\quad V^n(t; t,x,v) = v.  
\end{array}\right.
\end{equation}
Notice that Lemma~\ref{lm:confi} and Remark~\ref{rm:cond-confine} ensure that the characteristic $X(s)$ 
never reaches $\partial\Omega$.  
Since $K^{n-1}\in C^2([0,T]\times \Omega\times \R^2)$, we know that $(X^n, V^n)\in C^2([0,T];\R^3)$. 
We define the $n$-th iterate of the  distribution function by
\[
f^n(t,x,v) := f^0( X^n(0), V^n(0)).
\]
Then $f^n\in  C^2([0,T]\times \Omega\times \R^2)$ and it satisfies the initial value problem
\begin{equation}\label{Vlasov-n}
\left\{\begin{array}{rl}
&\partial_t f^n +\hat v_1 \partial_x f^n + K^{n-1}\cdot \nabla_v f^n =0,\\
&f^n(0, x,v) = f^0(x,v).  
\end{array}\right.
\end{equation}
Moreover, Lemma~\ref{lm:-support} shows that $f^n$ has compact support in the $x$ and $v$ variables, i.e.  $f^n\in  C^2_0([0,T]\times \Omega\times \R^2)$. Therefore, 
the functions
\[
 \rho^n(t,x) := \int_{\R^2} f^n(t,x,v)\, dv \quad\mbox{and}\quad  j^n(t,x) := \int_{\R^2} \hat v f^n(t,x,v)\, dv
 \]
are in $C^2_0([0,T]\times \Omega)$.
Next, we
define 
\begin{equation}\label{eq:E^n_1}
E^n_1(t,x) = \int_0^x \rho^n(t,y)\, dy +\lambda
\end{equation}
and 
 $E^n_2,\, B^n$ as the solution of 
\begin{equation}\label{ME-n}
\left\{\begin{array}{rl}
\partial_t E^n_2 = -\partial_x B^n -  j^n_2,\quad &\partial_t B^n = - \partial_x E^n_2,\\
E^n_2(0,x)= E^0_2(x),\, & B^n(0,x) =B^0(x),\,\\
E^n_2(t,x)_{\vert_{\partial \Omega}}=E^b_2(t,x),\, &B^n(t,x)_{\vert_{\partial \Omega}}=B^b(t,x). 
\end{array}\right.
\end{equation}
As in Section~\ref{sec:bound-fields}, we know that $E^n_2$ and $B^n$ must be given by the formulas in Lemma~\ref{representation-E_2-B} 
with $j_2$ being replaced by
$j^n_2$. We deduce that  $E^n_1, E^n_2, B^n\in C^2([0,T]\times \overline{\Omega})$.

Then it follows from the same reasoning leading  to Lemma~\ref{C^1-est} that  there exists a constant $C_T>0$ depending only on $k_0,\, T, \,\lambda,\, \|\Be\|_{C^1([\theta_1, 1-\theta_1])}$,  the $C^1$ norms of $f^0,\, E_2^0,\, 
 B^0$,   and the $C^1$ norms of $ 
E_2^b(\cdot,x),\, B^b(\cdot,x)$ on $[0,T]$ ($x=0,\, 1$) such that
\[
\|f^n\|_{C^1([0,T]\times \overline\Omega\times \R^2)} + \|E^n\|_{C^1([0,T]\times \overline\Omega)}+\|B^n\|_{C^1([0,T]\times \overline\Omega)}\leq C_T.
\]
Moreover, by following the arguments in \cite[Section~5.8]{G} we see that $\{(f^n, E^n_1, E^n_2, B^n)\}$ is a Cauchy sequence in the $C^1$ norm. Consequently,  there exist
$f\in  C^1([0,T]\times \overline{\Omega}\times \R^2)$ and
$E_1, E_2, B\in C^1([0,T]\times \overline{\Omega})$ such that $f^n \to f$, $E^n_1\to E_1$, $E^n_2\to E_2$, $B^n\to B$ uniformly for 
$t\in [0,T],\, x\in \overline{\Omega},\, v\in \R^2$, together with all their first derivatives. In particular, the function $f, E_2$ and $B$ satisfy the  initial and boundary
conditions \eqref{initial-condition}--\eqref{boundary-condition}. Note also that $f\in C^1_0([0,T]\times \Omega\times \R^2)$ since the $(x,v)$- support of $f^n$ is bounded uniformly in $n$ by Lemma~\ref{lm:-support}.

Passage to the limit in \eqref{Vlasov-n} yields the Vlasov equation. On the other hand, passage to the limit in \eqref{eq:E^n_1} and \eqref{ME-n} yields
\begin{equation*}
\left\{\begin{array}{rl}
&E_1(t,x) = \int_0^x \rho(t,y)\, dy +\lambda, \\
&\partial_t E_2 = -\partial_x B -  j_2,\quad \partial_t B = - \partial_x E_2.
\end{array}\right.
\end{equation*}
Thus $(f, E_1, E_2, B)$ is a $C^1$ solution to the problem  \eqref{VE}--\eqref{boundary-condition} in the time interval $[0,T]$. Due to the 
arbitrariness of $T$ and the uniqueness of $C^1$ solutions presented earlier, we infer that the problem  
\eqref{VE}--\eqref{boundary-condition} admits a  global classical solution $(f, E, B)$ with $f\in  C^1([0,\infty)\times \overline{\Omega}\times \R^2)$ and 
$E, B\in C^1([0,\infty)\times \overline{\Omega})$. Moreover, $f\in C^1_0([0,T]\times \Omega\times \R^2)$ for every $T>0$.

We finally note that 
the non-negativity of 
the  solution $f$ is inherited from that of $f^0$ as $f$ is constant along the characteristics.    \end{proof}

\bibliographystyle{plain}

\medskip
\medskip

\end{document}